\documentclass[a4paper,twoside,12pt]{article}
\usepackage{amssymb,amsfonts,amsmath,amsthm,latexsym}
\usepackage[pdftex,bookmarks,colorlinks=false]{hyperref}
\usepackage[hmargin=1.2in,vmargin=1.2in]{geometry}
\usepackage[auth-sc-lg,affil-sl]{authblk}
\usepackage{graphicx,caption,subcaption,centernot}
\usepackage[mathscr]{euscript}
\newtheorem{thm}{Theorem}[section]

\newtheorem{defn}[thm]{Definition}
\newtheorem{lem}[thm]{Lemma}
\newtheorem{prob}[thm]{Problem}
\newtheorem{prop}[thm]{Proposition}
\newtheorem{rem}[thm]{Remark}

\numberwithin{equation}{section}

\def\ni{\noindent}
\def\N{\mathbb{N}}
\def\J{\mathbb{J}}
\def\H{\mathbb{H}}

\def\C{\mathcal{C}}

\title{\textbf{\sc  The $b$-Chromatic Number of Certain Graphs and Digraphs}}

\author{Johan Kok}
\affil{\small Tshwane Metropolitan Police Department\\ City of Tshwane, Republic of South Africa \\ {\tt kokkiek2@tshwane.gov.za}}

\author{Naduvath Sudev}
\affil{\small Department of Mathematics\\ Vidya Academy of Science \& Technology \\ Thalakkottukara, Thrissur - 680501, India.\\ {\tt sudevnk@gmail.com}}

\date{}

\begin{document}
\maketitle

\begin{abstract}
The $b$-chromatic number of a graph is the maximum number $k$ of colors that can be used to color the vertices of $G$, such that we obtain a proper coloring and each color $i$, with $1\le i\le k$, has at least one representant $x_i$ adjacent to a vertex of every color $j$, $1\le j\ne i \le k$. The studies on $b$-chromatic number attracted much interest since its introduction. In this paper, we discuss the $b$-chromatic number of certain classes of graphs and digraphs. The notion of a new general family of graphs called the Chithra graphs of a graph $G$ are also introduced in this study and the corresponding $b$-chromatic numbers are studied.
\end{abstract}

\ni \textbf{Keywords:} $b$-Chromatic number, set-graphs, edge-set graphs, Jaco graph, ornated graphs, Rasta graph, Chithra graph.
\vspace{0.2cm}


\section{Introduction}

For general notation and concepts in graph and digraph theory, we refer to \cite{BM1,CL1,GY1,DBW}. Except for Jaco and Ornated graphs, all graphs mentioned in this paper are simple, non-trivial, connected, finite and undirected graphs.

Let us denote the set of $k$ colors by $\C = \{c_1,c_2, c_3, \ldots , c_k\}$. Then, \textit{a proper $k$-coloring} of a graph $G$ is a function $f:V(G)\to \C$ such that $f(u)\ne f(v); \forall uv \in E(G)$. The \textit{color class} $V_{c_i}$ is the subset of $V(G)$ that are assigned to the color $c_i$. The \textit{chromatic number} $\chi(G)$ is the minimum integer $k$ for which $G$ admits proper $k$-coloring.

A proper $k$-coloring $\C$ of a graph $G$ is called a $b$-coloring if for every color class $\C_i$, there is a vertex with color $i$ which has at least one neighbour in every other color classes. Such a vertex $v$ in $G$ is called a \textit{b-vertex} of $G$. 

The \textit{$b$-chromatic number}, denoted $ \varphi(G)$, of a graph $G$ had been introduced in \cite{IM1} as the largest integer $k$ for which $G$ admits a $b$-coloring using $k$ colors. If $G$ admits a $b$-coloring, then it is called \textit{b-colorable graph}.

In other words, the $b$-chromatic number of a graph $G$ is defined (see \cite{EK1}) as follows. 

Let $G$ be a graph on $n$ vertices, say $v_1, v_2,v_3,\ldots, v_n$. The \textit{$b$-chromatic number} of $G$ is defined as the maximum number $k$ of colors that can be used to color the vertices of $G$, such that we obtain a proper coloring and each color $i$, with $1\le i\le k$, has at least one representant $x_i$ adjacent to a vertex of every color $j$, $1\le j\ne i \le k$.

For more literature on $b$-chromatic number of graphs, we refer to  \cite{EK1,IM1,KM1,VI1,VV1}. Motivated from these studies, in this paper, we discuss the $b$-chromatic number of certain types of graphs and digraphs 

\ni The following is an important and useful result provided in \cite{KM1}. 

\begin{prop}\label{Prop-1.2}
{\rm \cite{KM1}} For the the complete graph $K_n, n \ge 1$, the path $P_n, n \ge 2$, the cycle $C_n, n\ge 3$ we have
\begin{enumerate}\itemsep0mm
\item[(i)] $\varphi(K_n)= n; ~\forall ~n \in \N$,
\item[(ii)] $\varphi(P_n)= 2;~ n=2,3$ and $\varphi(P_n)=3,  n\ge 4$.
\item[(iii)] $\varphi(C_n)=3, \forall n\ne 4$ and $\varphi(C_4)=2$.
\item[(iv)] $\varphi(K_{m'n})=2$.
\end{enumerate}
\end{prop}


\section{New Directions}

Recently, the families of graphs namely, Jaco graphs, linear Jaco graphs, ornated graphs, Rasta graphs, set-graphs, edge-set graphs and edge-joint graphs were introduced in different studies in \cite{KS4,KS5,KC2S,KSM,KSK,KSC}. 

We present a slightly modified definition of the $b$-chromatic number, perhaps slightly different from those found in the current literature will be presented hereafter the $b$-chromatic number for these families of graphs as  follows. 

Consider a proper $k$-coloring of a graph $G$ and denote the set of $k$ colors by $\C = \{c_1,c_2, c_3, \ldots , c_k\}$. Define the subsets $V_{c_i}$, of $V(G)$ by $V_{c_i} = \{v_j: v_j \mapsto c_i, v_j \in V(G), c_i \in \C\}$, where $1 \le i \le k$. The subset $V_{c_i}$ is called a \textit{color class} of $G$. Clearly, the collection of all such color classes form a partition of $V(G)$. 

The largest possible value of $k \in \N$ such that there exists a proper $k$-coloring such that there exist two adjacent vertices $v_s \in V_{c_i}$ and $v_t \in V_{c_j}$, for all distinct pairs of colors $c_i, c_j$, is called the $b$-chromatic number of $G$, denoted by $\varphi(G)$. Such a coloring of the graph $G$ is called a \textit{$b$-coloring} of $G$.

\subsection{$b$-Chromatics Number of Linear Jaco Graphs}

Let $\N$ denotes the set of all natural numbers and $\N_0$ denotes the set of all non-negative integers. The notion of the family of linear Jaco graphs has been introduced in \cite{KSK} as given below. 

\begin{defn}{\rm 
\cite{KSK} Let $f(x)=mx+c;~ x,m\in \N$, $c\in  \N_0$. The \textit{family of infinite linear Jaco graphs}, denoted by $\{J_\infty(f(x))$, is defined as a directed graph such that 
\begin{enumerate}\itemsep0mm
\item[(i)] $V(J_\infty(f(x)))=\{v_i: i \in \N\}$,
\item[(ii)] $A(J_\infty(f(x))) \subseteq \{(v_i, v_j): i, j \in \N, i< j\}$, 
\item[(iii)] $(v_i,v_ j) \in A(J_\infty(f(x)))$ if and only if $(f(i)+i)-d^-(v_i)\ge j$.
\end{enumerate} 
}\end{defn}

The function $f(x)$, we mention in this section is a linear function of the form  $f(x)=mx+c;~ x,m \in \N, c\in \N_0$. Then, a finite linear Jaco graph can be defined as follows.

\begin{defn}{\rm \cite{KSK}
The \textit{family of finite linear Jaco graphs}, denoted by $\{J_n(f(x))$, is defined by $V(J_n(f(x))=\{v_i:i\in \N, i\le n \}$, $A(J_n(f(x))) \subseteq \{(v_i, v_j): i,j \in \N, i< j \le n\}$ and $(v_i,v_ j) \in A(J_n(f(x)))$ if and only if $(f(i) + i) - d^-(v_i) \ge j$.
}\end{defn} 

The underlying graph of a finite linear Jaco graph is denoted by $J^{\ast}_n(f(x))$, and is referred simply as the linear Jaco graph.

The vertices of a finite linear Jaco graph $J_n(f(x))$ with degree $\Delta(J_n(f(x)))$ is called \textit{Jaconian vertices} of $J_n(f(x))$ and the set of vertices with maximum degree is called the {\em Jaconian set} of the linear Jaco graph $J_n(f(x))$, denoted, $\J(J_n(f(x)))$ or simply $\J_n(f(x))$ (see \cite{KSK}). The lowest numbered (indexed) Jaconian vertex is called the \textit{prime Jaconian vertex} of a linear Jaco graph.

If $v_i$ is the prime Jaconian vertex of the finite linear Jaco graph $J_n(f(x))$, the complete subgraph induced by the vertex subset $\{v_{i+1},v_{i+2}, \ldots, v_n\}$ is called the {\em Hope subgraph} of $J_n(f(x))$, denoted by $\H(J_n(f (x)))$ or $\H_n(f (x))$.


Invoking the above notions, we establish the following result on the $b$-chromatic number of a finite linear Jaco graph.

\begin{thm}\label{T-BCN-JG}
For $n\ge 2$, the $b$-chromatic number of a linear Jaco graph $J^{\ast}_n(f(x))$, with prime Jaconian vertex $v_i$ is given by $\varphi(J^{\ast}_{n\ge 2}(f(x))) = (n-i) + 1$.
\end{thm}
\begin{proof}
If vertex $v_1$ is the prime Jaconian vertex of $J_{n}(f(x))$, then $J^{\ast}_{n}(f(x))\cong K_n$ and the result follows from Proposition \ref{Prop-1.2}. If vertex $v_2$ is the prime Jaconian vertex, then the Hope graph induced by the set $\{v_3, v_4, v_5, \ldots v_n\}$ is complete. Color these vertices as follows. $v_3\mapsto c_1, v_4\mapsto c_2, v_5\mapsto c_3,\ldots v_n \mapsto c_{n-2}$.  Also, color $v_1\mapsto c_{n-2}$ and $v_2\mapsto c_{n-1}$. Clearly, this proper $k$-coloring with $k=(n-2)+1=n-1$, is a maximum $k$-coloring  and hence, $\varphi(J^{\ast}_{n\ge 2}(f(x))) = (n-i)+1$. 

Now, assume the result holds for the prime Jaconian vertex $v_i$, $ 3 \le i \le k$ for the corresponding Jaco graph $J^{\ast}_m(f(x))$. Now, consider the Jaco graph $J^{\ast}_{m+1}(f(x))$. Then, we have to consider the following cases.

\ni \textit{Case-1:} If $v_i$ is the prime Jaconian vertex of $J^{\ast}_{m+1}(f(x))$, then the corresponding Hope graph is $\H(J^{\ast}_{m+1}(f(x))) = K_{(m-i) +1}$. We color these vertices as $v_{i+1}\mapsto c_0, v_{i+2}\mapsto c_1, v_{i+3}\mapsto c_2,\ldots v_{m+1} \mapsto c_{(m+1)-(i+1)}$. Let $t=(m+1)-i$. Color the vertices $v_i \mapsto c_{((m+1) - i)}, v_{i-1}\mapsto c_{m-i}, v_{i-2}\mapsto c_{(m-i) -1}, \ldots , v_{2i- n} \mapsto c_0$. Repeat the coloring procedure iteratively until $v_1\mapsto c_i$, $ 0\le i \le t$ is colored. Note that $\lceil \frac {i}{t}\rceil$ iterations are required to color all the vertices completely. Clearly, this is a  maximal proper coloring. 

Assume without loss of generality that it is possible to color $v_1 \mapsto c_{t+1}$. Then, if $v_{i+1}\mapsto c_j$, and since $v_1v_{i+1} \notin E(J^{\ast}_{m+1}(f(x)))$ the coloring is not a $b$-coloring.

\ni \textit{Case-2:} If $v_i$ is not the prime Jaconian vertex of $J^{\ast}_{m+1}(f(x))$ then by necessity vertex $v_{i+1}$ is the prime Jaconian vertex. The Hope graph, $\H(J^{\ast}_{m+1}(f(x))) = K_{(m+1)- (i+1)} = K_{m-i}$. In this case also the result follows by proceeding as explained in Case-1 alternatively, as applicable in the assumption.

Therefore, in both cases we have $\varphi (J^{\ast}_{m+1}(f(x))) = ((m+1) - i) +1$ or $(m-i) + 1$. Hence, the given result follows by mathematical induction.
\end{proof}

\subsection{$b$-Chromatics Number of Ornated Graphs}

The class of ornated graphs is a family of directed non-simple finite graphs which is defined in \cite{KSM} as given below.

\begin{defn}{\rm 
\cite{KSM} An \textit{ordered string}, denoted $s_l$, is defined as an $l$-tuple $(a_j)^l_{j=1}$ where $a_j \in \N_0$, $1 \le j \le l$ for $j,l \in \N$. For brevity, we write $s_l = (a_j); 1 \le j \le l$.
}\end{defn}

\begin{defn}{\rm 
\cite{KSM} Let $n$ be a positive integer and $s_l = (a_j); 1 \le j \le l$ be an ordered string of non-negative integers. An \textit{ornated graph} on $n$ vertices, associated with an ordered string $s_l$ is denoted by $O_n(s_l)$ and is defined as a directed graph with vertex set $V=V(O_n(s_l)) = \{v_i:i\in \N,i\le n\}$ and the arc set $A=A(O_n(s_l)) \subseteq V \times V$ such that for $i,j \in \N$, $(v_i,v_j) \in A$ if and only if $(i+a_t) \ge j, i< j$, for odd indices $t$ and $(i-a_s) \le j,~i>j$, for even indices $s$.
}\end{defn}

The underlying graph of $O_n(s(l))$ is denoted $O^{\ast}_n(s_l)$. For determining $\varphi(O^{\ast}_n(s_l))$ we only have to determine $\varphi(O^{\ast}_n(s'_l))$ called the \textit{maximal reach} subgraph with $s'_l= (a_i); a_i \ge a_j$, $1 \le j \le n$. In view of all these facts, we determine the $b$-chromatic number of ornated graphs in the following theorem.

\begin{thm}
For an Ornated graph $O_n(s_l)$, $a_j > 0$ for some $i$, and $a_i \ge a_j$, $1 \le j \le n$, we have
\begin{equation*} 
\varphi(O^{\ast}_n(s_l)) =
\begin{cases}
n, & \text{if $n \le a_i + 1$,}\\  
a_i + 2, & \text{if $n > a_i + 1$}.
\end{cases}
\end{equation*} 
\end{thm}
\begin{proof}
First assume that $n \le a_i + 1$. Then, since $O_n(s'_l)$, $s'_l =(a_i)$ is the maximum reach subgraph of $O_n(s_l)$, the graph $O^{\ast}_n(s'_l)$ is a complete graph on $n$ vertices. Therefore, by Proposition \ref{Prop-1.2}, $\varphi(O^{\ast}_n(s_l))=O^{\ast}_n(s'_l)=n$.

Next, assume that $n \ge a_i + 2$. Initially, color the vertices as $v_1 \mapsto c_1, v_2 \mapsto c_2, v_3 \mapsto c_3,\ldots ,v_{a_i + 2} \mapsto c_{a_i+2}$. Since the edge $v_1v_{a_i+2}\notin E(O^{\ast}_n(s'_l))$, we must color $v_{a_i+3} \mapsto c_1$. Then, by mathematical induction it follows that $\varphi(O^{\ast}_n(s_l)) = \varphi(O^{\ast}_n(s'_l)) = a_i + 2$.
\end{proof}

\subsection{$b$-Chromatics Number of  Rasta Graphs}

\ni The notion of Rasta graphs had been introduced in \cite{KS5}, which can be constructed as explained below.

\textbf{Constructing a Rasta graph:} Consider a t-term sum set $\{t_1, t_2, t_3, ..., t_l\}$ with $t_1 > t_2 > t_3 > ... > t_l>1$. 

\begin{enumerate}\itemsep0mm
\item[(S-1)]  Consider $t_1$ vertices in the left column (the first column) and $t_2$ vertices in the Right column (the second column) and construct $K_{t_1, t_2}$.
\item[(S-2)]  Add the third column of $t_3$ vertices and add the edges of $K_{t_2, t_3}$.
\item[(S-3)]  Repeat Step 2 iteratively up to $t_\ell$.
\item[(S-4)]  Exit.
\end{enumerate}

It is noted in \cite{KS5}that a Rasta graph is the underlying graph of a directed graph defined as follows.

\begin{defn}{\rm \cite{KS5}
For a $l$-term sum set $\{t_1, t_2, t_3, \ldots, t_l\}$ with $t_1 > t_2 > t_3 > \ldots > t_l>1$, we can define the directed graph $G^{(l)}$ with vertices $V(G^{(l)}) =\{v_{i,j}: 1\le j \le t_i, 1\le i \le l\}$ and the arcs, $A(G^{(l)})=\{(v_{i,j},v_{(i+1),m}):1\le i \le (l-1), 1\le j \le t_i$ and $1 \le m \le t_{(i+1)}\}$. 
}\end{defn}

\begin{thm}
For a Rasta graph $R$ corresponding to the underlying graph of $G^{(\ell)}$ we have $\varphi(R) = 2$.
\end{thm}
\begin{proof}
We know that $\varphi(K_{n,m}) = 2$ (see \cite{KM1}). For the complete bi-partite graph $K_{t_1, t_2}$ color the vertices in the left column, $c_1$ and those in the right column, $c_2$. From the construction of a Rasta graph mentioned above, we can observe that no vertex in the first column (that is, in the leftmost column) is adjacent to any vertex in the third column and hence the vertices in the third column can be be colored $c_1$. Similarly, the vertices in the second column and those in the fourth column are not adjacent and hence the vertices in the fourth column can be colored using the color $c_2$. Proceeding like this, the vertices at the odd position columns have the color $c_1$ and the vertices at the even position columns have the color $c_2$. Therefore, we have $\varphi(R)=2$.
\end{proof}

\subsection{Chithra Graphs of a Graph}

In this section, we introduce a new family of graphs, namely \textit{Chithra graphs\footnote{The first author wishes to dedicate these graphs to a research colleague, Mrs. Chithra Sudev.}}, of a given graph G as follows. 

Consider the set $\mathcal{V}(G)$ of all subsets of the vertex set, $V(G)$. Select any number, say $k$, non-empty subsets $ W_i\in \mathcal V(G)$, $\bigcup^{k}_{i=1}W_i=V(G)$, with repetition of selection allowed. For the additional vertices $u_1, u_2, u_3, \ldots , u_k$, add the additional edges $u_iv_j$, $\forall v_j \in W_j$. This new graph is called a \textit{Chithra graph} of the given graph $G$ and the family of Chithra graphs of the graph $G$ is denoted by $\mathfrak{C}(G)$.  

Note that in general the empty subset of $V(G)$ may be selected  as repetition of selection is permitted. However, the empty subset does not represent an additional vertex. This argument implies that $G \in \mathfrak{C}(G)$ as well. 

Also, note that the Chithra graphs can be constructed from non-connected (disjoint) graphs and for edgeless (null) graphs. It means that for the edgeless graph on $n$ vertices denoted $\mathcal{N}_n$, the complete bi-partite graph, $K_{n,m}\in \mathfrak{C}(\mathcal{N}_n)$. Equally so, $K_{n,m}\in \mathfrak{C}(\mathcal{N}_m)$. Constructing paths from an edgeless graph are excluded. 

A number of well-known classes of graphs are indeed Chithra graphs of some graph $G$. Small graphs such as sun graphs, sunlet graphs, crown graphs and helm graphs are all Chithra graphs of some graph $G$.

The $b$-chromatic number of the Chithra graphs of the given graphs is determined in the following theorem. 
 
\begin{thm}\label{T-bCN-CG}
Let $G$ be a graph and $G'\in \mathfrak{C}(G)$ and $G' \ne P_n, n \ge 4$. Then $\varphi(G')=\varphi(G)+ 1$.
\end{thm}
\begin{proof}
Clearly, if the color $c_{\varphi +1}$ is allocated to each additional vertex $u_i$, $1 \le i \le k$, the coloring $\C \cup \{c_{\varphi+1}\}$ is a proper coloring of $G'$. Since the edges $u_iu_j \notin E(G')$, for all $i,j$, no vertex $u_i$ can be colored other than with $c_{\varphi +1}$. Hence, $\C \cup \{c_{\varphi+1}\}$ is a $b$-coloring of $G'$. Therefore, $\varphi(G')=\varphi(G)+1$.
\end{proof}

\ni Some trivial applications of Theorem \ref{T-bCN-CG} are the following.

\begin{enumerate}\itemsep0mm
\item[(i)] Since $\varphi(K_1) =1$, and $P_3 \in \mathfrak{C}(K_1)$, it follows that $\varphi(P_3)=2$.
\item[(ii)] Since $\varphi(C_n) =3$, $n\ge 5$, and the sunlet graph $S_n \in \mathfrak{C}(C_n)$, it follows that $\varphi(S_n) = 4$ (see \cite{VV1}).
\item[(iii)] Since $\varphi(C_n)=3,~n\ge 5$, and the wheel graph $W_{n+1} \in \mathfrak{C}(C_n)$, it follows that $\varphi(W_{n+1}) = 4$.
\item[(iv)] A \textit{sun graph} $S_{2n}$ (see \cite{GIC}) has $V(S_{2n})=\{d_i:1\le i\le 2n\}$ and $E(S_{2n}) = E(K_n)\bigcup \{d_id_{i+n}:1\le i\le n\} \bigcup \{d_{i+n}d_{(i+1) +n}:1 \le i \le n-1\} \bigcup \{d_{2n}d_1\}$. Now, since a sun graph $S_{2n} \in \mathfrak{C}(K_n)$, it follows that $\varphi(S_{2n})=n+1$.
\item[(v)]  Since a helm graph $H_{2n+1}$ is obtained from a wheel $W_{n+1}$ by attaching one pendant vertex to each vertex of the outer cycle $C_n$, it follows that $H_{2n+1} \in \mathfrak{C}(W_{n+1})$. Therefore, $\varphi(H_{2n+1})=5$.
\end{enumerate}

Consider a graph $G$ on $n$ vertices. If the maximum number of mutually non-adjacent vertices in $G$ is $k$, label them $u_1, u_2, u_3,\ldots,u_k$. We recall that the open neighbourhood of a vertex $v\in V(G)$ is denoted by $N_G(v)$.

\ni In view of the above notions and fats, we have the following result.

\begin{lem}
If for a graph $G$ on $n$ vertices with a maximum of $k$ mutually non-adjacent vertices, $U = \{u_1, u_2, u_3,\ldots , u_k\}$ we have $\bigcup^{k}_{i=1}N_G(u_i)=V(G)$, then $G \in \mathfrak{C}(G-U)$.
\end{lem}
\begin{proof}
Since $G$ is connected, $N_G(u_i)\ne \emptyset$, $\forall u_i$. Also, each $N_G(u_i) \subseteq V(G)$ and $\bigcup^{k}_{i=1}N_G(u_i) = V(G)$. Hence, from the construction of a Chithra graph of $G$, it is evident that $G \in \mathfrak{C}(G-U)$.
\end{proof}

\ni In other words it can be said that if $\sum\limits^{k}_{i=1}d_G(u_i) \ge n-k$ then $G \in \mathfrak{C}(G-U)$.

\subsection{$b$-Chromatics Number of Set-Graphs}

\ni The notion of a set-graph was introduced in \cite{KC2S} as explained below.

\begin{defn}{\rm 
\cite{KC2S} Let $A^{(n)} = \{a_1,a_2,a_3,\ldots , a_n\}$, $ n \in \N$ be a non-empty set and the $i$-th $s$-element subset of $A^{(n)}$ be denoted by $A^{(n)}_{s,i}$. Now, consider $\mathcal S = \{A^{(n)}_{s,i}: A^{(n)}_{s,i} \subseteq A^{(n)}, A^{(n)}_{s,i} \ne \emptyset \}$. The \textit{set-graph} corresponding to set $A^{(n)}$, denoted $G_{A^{(n)}}$, is defined to be the graph with $V(G_{A^{(n)}}) = \{v_{s,i}: A^{(n)}_{s,i} \in \mathcal S\}$ and $E(G_{A^{(n)}}) = \{v_{s,i}v_{t,j}: A^{(n)}_{s,i} \cap A^{(n)}_{t,j} \ne \emptyset\}$, where $s\ne t$ or $i\ne j$.
}\end{defn}

The largest complete graph in the given set-graph $G_{A^{(n)}}$, $n \ge 2$  is $K_{2^{n-1}}$ and the number of such largest complete graphs in the given set-graph $G_{A^{(n)}}$, $n \ge2$ is provided in the following proposition.

\begin{prop}\label{Prop-KGSG}
{\rm \cite{KC2Se}} The set-graph $G_{A^{(n)}}, n\ge 2$ has exactly $2^{n-1}$ largest complete graphs $K_{2^{n-1}}$.
\end{prop}

Invoking the above notion of set-graphs and Proposition \ref{Prop-KGSG}, we can establish the following result on the $b$-chromatic number of a set-graph.

\begin{thm}
A set-graph $G_{A^{(n)}}$, $n \ge 2$ has $\varphi(G_{A^{(n)}}) = 2^{n-1} +1$.
\end{thm}
\begin{proof}
By Proposition \ref{Prop-KGSG}, we know that a largest complete subgraph of the set-graph $G_{A^{(n)}}$, $n \ge 2$ is $K_{2^{n-1}}$. Furthermore, there exists such a complete graph, say $K^{\ast}_{2^{n-1}}$, such that the vertices $v_{1,i} \notin V(K^{\ast}_{2^{n-1}})$, $1\le i \le n$. Let the $b$-coloring of $K^{\ast}_{2^{n-1}}$ be the set $\C^{\ast} = \{ c_1,c_2,c_3,\ldots , c_{2^{n-1}}\}$. 

Since the vertices in the set $U =\{v_{1,i}: 1\le i \le n\}$ are mutually non-adjacent and $\bigcup^{n}_{i=1}N_G(v_{1,i}) = V(K^{\ast}_{2^{n-1}})$, the subgraph $K^{\ast}_{2^{n-1}} + U$ is a Chithra graph of $K^{\ast}_{2^{n-1}}$.  Hence, $\varphi(K^{\ast}_{2^{n-1}}+U)=2^{n-1}+1$. In the remaining subset of  vertices $V'=V(G_{A^{(n)}}) - V(K^{\ast}_{2^{n-1}} + U)$ each vertex $v_m\in V'$ is adjacent to at least one other vertex $v_{m'} \in V'$. The latter argument implies that new colors cannot be added and all vertices $v_m \in V'$ must be colored from $\C^{\ast}$, which is always possible because $|V'|=(2^n-1)-2^{n-1}-n<\frac{1}{2}(2^n-1)$. Hence, we have $\varphi(G_{A^{(n)}}) = 2^{n-1}+1$.
\end{proof}

\ni Another class of graphs similar to the set-graphs is the class of edge-set graphs which has been introduced in \cite{KSC} as follows. 

\begin{defn}{\rm 
\cite{KSC} Let $G(V,E)$ be a non-empty finite graph with $|E|= \epsilon$ and $\mathcal {E} = \mathcal {P}(E) - \{\emptyset\}$, where $\mathcal {P}(E)$ is the power set of the edge set $E(G)$. For $1 \le s \le \epsilon$, let $\mathcal{S}$ be the collection of all $s$-element subsets of $E(G)$ and $E_{s,i}$ be the $i$-th element of $\mathcal{S}$. Then the edge-set graph corresponding to $G$, denoted by $G_{E^{(\epsilon)}}$ or $\mathcal{G}_G$, is the graph with the following properties.

\begin{enumerate}\itemsep0mm
\item[(i)] $|V(\mathcal{G}_G)|=2^\epsilon-1$ so that there exists a one to one correspondence between $V(\mathcal{G}_G)$ and $\mathcal{E}$.

\item[(ii)] Two vertices, say $v_{s,i}$ and $v_{t,j}$ in $\mathcal{G}_G$ are adjacent if some elements (edges of $G$) in $E_{s,i}$ is adjacent to some elements of $E_{t,j}$ in $G$.
\end{enumerate}
}\end{defn}

Clearly, the edge-set graph is dependent on both the number of edges and the structure of $G$. It is known that $\varphi(G) \le \Delta(G)+1$  (see \cite{IM1,KM1}). In \cite{KSC} it is proved that the edge-set of $G$ on $n$ vertices is a complete graph if and only if $G$ is a star graph. Hence, it can be said that $\varphi(\mathcal{G}_G) \le \varphi(K_n)$.  Because the vertices of the edge-set graph is determined by the non-empty subsets of the number of edges of $G$ while the adjacency structure of $\mathcal{G}_G$ depends on the adjacency structure of $G$, the bound $\varphi(\mathcal{G}_G)\le n$ is optimal. For all non-star graphs, this bound improves to $\varphi(\mathcal{G}_G)\le n-1$.

\subsection{Edge-joint Graphs}

The notion of an edge-joint graph is simple. It is an operation between two graphs $G$ and $H$ and requires the addition of an arbitrary edge to join the graphs into a connected structure (see \cite{KS4}).

\begin{defn}{\rm 
\cite{KS4} The\textit{ edge-joint} of two graphs $G$ and $H$ is the graph obtained by adding the single arbitrary edge $vu$, $v \in V(G)$, $u \in V(H)$, and is denoted, $G\rightsquigarrow_{vu}H$. 
}\end{defn}

Clearly, the edge-joint operation establishes minimal connectivity between two graphs and results in the existence of at least one bridge in $G\rightsquigarrow_{vu}H$. For parameters which are reliant on the orientation of a graph such as determining the brush number of a graph it is seen that the commutative property does not necessarily hold. 

Certain structural properties are either enhanced or destroyed through this operation. By edge joint operations, connectivity in the graphs is enhanced where as the existence of Hamilton cycles is destroyed. Furthermore, if both graphs $G$ and $H$ are without bridges, all results for a graph with bridges can now be applied for the edge joint graph.

In view of these facts, the $b$-chromatic number of an edge joint graph is determined in the following theorem.

\begin{thm}
Let $\varphi(G)\ge \varphi(H)$ for two graphs $G$ and $H$, which are not paths with $2$ or $3$ vertices. Then, $\varphi(G\rightsquigarrow_{vu}H) = \varphi(H\rightsquigarrow_{uv}G) = \varphi(G)$.
\end{thm}
\begin{proof}
Assume that the $b$-colorings of two graphs $G$ and $H$ are respectively given by $\C =\{ c_1, c_2, c_3,\ldots , c_{\varphi(G)}\}$ and $\C' = \{c'_1, c'_2, c'_3,\ldots , c'_{\varphi(H)}\}$ and $c_i \ne c'_j, \forall i,j$. Since $|\C'| \le |\C|$, it is always possible to construct a one-on-one mapping $f(c'_i) \mapsto c_j \iff f^{-1}(c_j) = c'_i$.  Denote the mapping $F(\C') \mapsto \C$. Obviously, the vertex $u$ cannot be colored with a color $c^{\ast} \notin \C$ and $c^{\ast} \notin \C'$ else $\C' \cup \{c^{\ast}\}$ is a $b$-coloring of $H$ larger than the $b$-coloring $\C'$ which is a contradiction.

Since $|V(H)|\ge 2$, there exists a vertex $u'\ne u$, $u',u\in V(H)$. Clearly, vertex $u'$ cannot be colored with $c^{\ast} \notin \C$, else the $b$-coloring property is destroyed. Hence, the vertices of $H$ must be colored by the mapping $F(\C') \mapsto \C$. Thus, $\varphi(G\rightsquigarrow_{vu}H) = \varphi(H\rightsquigarrow_{uv}G) = \varphi(G)$. 
\end{proof}

\begin{rem}{\rm 
For  the graphs $G=P_n, ~n\le 3$ and $H=P_m,~m\le 3$, we can apply Proposition \ref{Prop-1.2} to find the $b$-chromatic number of the edge joint graph of the graphs $G$ and $H$.}
\end{rem}

\section{Conclusion}

In this paper, we have discussed the $b$-chromatic number of certain classes of graphs and digraphs. 

The concept of Chithra graphs allows another interesting observation namely, the smallest graph from which the complete graph $K_n,~n\ge 2$ can be constructed recursively is $K_1$. This is true because $K_n \in \mathfrak{C}(K_n),~K_n \in \mathfrak{C}(K_{n-1}), K_{n-1} \in \mathfrak{C}(K_{n-2}), K_{n-2} \in \mathfrak{C}(K_{n-3}),\ldots , K_2 \in \mathfrak{C}(K_1)$. In other words, $\mathfrak{C}(K_n) \subset \mathfrak{C}(K_{n-1}) \subset \mathfrak{C}(K_{n-2}) \subset \ldots \subset \mathfrak{C}(K_1)$. In this sense, $K_1$ is called the \textit{root Chithra graph} of $K_n$. 

For the sun graph $S_{2n}$, the root Chithra graph is $K_1$ as well. For the sunlet graph $S_n$ the cycle $C_n$ is the root Chithra graph. Note that the root Chithra graph may not be unique in general. For example, a root Chithra graph for the complete bi-partite graph $K_{n,m}$ may be $\mathcal{N}_n$ or $\mathcal{N}_m$.

Finally, for some graphs $G$, the root Chithra graph is $G$, itself, of which the Jaco graph $J^{\ast}_n(f(x))$ serves as an example. Hence, a root Chithra graph exists for all graphs because amongst the families of Chithra graphs of $G$, it is true that $G\in \mathfrak{C}(G)$ as well. It remains an open problem to determine the root Chithra graph of general classes of graphs.

With regards to edge-set graphs we have that for all non-star graphs $\varphi(\mathcal{G}_G)\le n-1$. This observation leaves a wide scope for research into the parameter $\varphi(\mathcal{G}_G)$ for specialised classes of graphs with well-defined adjacency structures such as paths, cycles and more complex graphs such as complete graphs and Jaco graphs.

\ni Some other the open problems, identified during our studies are the following.

\begin{prob}{\rm 
Determine the root Chithra graph of various other classes of graphs.}
\end{prob} 
 
\begin{prob}{\rm 
Determine $\varphi(\mathcal{G}_G)$ for special classes of graphs with well-defined adjacency structures.}
\end{prob}

\end{document}